\documentclass{article}
\usepackage{graphicx}
\usepackage{amsmath,amsthm}
\usepackage{amsfonts}
\usepackage{amssymb}
\usepackage{latexsym}
\usepackage{tikz}
\usepackage{fancyvrb}
\usepackage{datetime}
\usepackage{color}
\usepackage[colorlinks]{hyperref}

\newtheorem{lemma}{Lemma}[section]
\newtheorem{proposition}{Proposition}[section]

\newtheorem{remark}{Remark}
\newtheorem{theorem}{Theorem}[section]

\usepackage{graphicx}
\usepackage{amsmath}
\usepackage{amsfonts}
\usepackage{amssymb}
\usepackage{latexsym}
\usepackage{cite}
\begin{document}

\begin{center}
{\Large{The spectral determinations of the multicone graphs  $ K_w\bigtriangledown P$}}
\medskip

{Ali Zeydi Abdian\footnote{Lorestan University, College of Science, Lorestan, Khoramabad, Iran; e-mail: aabdian67@gmail.com; azeydiabdi@gmail.com}
} \end{center}

\begin{abstract}
 A multicone graph is obtained from
 the join of a clique and a regular graph. Let  $ P $ and $K_w$ denote the  Petersen  graph  and a complete graph on $ w $ vertices, respectively.  In this paper, we show that  multicone graphs  $ K_w\bigtriangledown P$ are determined by their signless Laplacian spectra, their Laplacian spectra, their complement with respect to signless Laplacian spectra, and their adjacency spectra.  
\end{abstract}

\section{Introduction}
Let $G = (V,E)$ be a simple graph with the vertex set $ V = V (G) = \left\{ {v_1, \cdots  , v_n} \right\}$ and the edge set $E = E(G) = \left\{ {e_1, \cdots  , e_m} \right\}$. Denote by $d(v)$ the degree of vertex $ v $. All notions on graphs that are not defined here may be found in \cite{LP, Ba, B, CRS, TABN, W}.  The join of two graphs $ G $ and $ H $ is a graph formed from disjoint copies of $ G $ and $ H $ by connecting any vertex of $ G $ to any vertex of $ H $.  The join of two graphs $ G $ and $ H $ is denoted by $ G\bigtriangledown H $.  Let $A(G)$ be the $(0,1)$-adjacency matrix of graph $ G $.  Let $ q_1,q_2, \cdots, q_n$ be the distinct eigenvalues of $G$ with multiplicities $m_1, m_2, \cdots , m_n$, respectively. The multi-set of eigenvalues of $ Q(G) $ is called the signless Laplacian spectrum of $G$. The matrices $ L(G) =D(G)-A(G) $ and $ Q(G)= D(G)+A(G) $ are called the Laplacian matrix and the signless Laplacian matrix of $ G $, respectively, where $D(G)$ is the degree matrix. Note that $D(G)$ is diagonal. The multi-set $ {\rm{Spec}}_{Q}(G) = \left\{ { [q_1]^{ m_1 } , q_2]^{ m_2 } , \cdots, [q _n]^{ m_n } } \right\}$ of eigenvalues of $ Q(G) $ is called the signless Laplacian spectrum of $G$, where $ m_i $ denote the multiplicity of
$ \lambda_i $ and $ q_1\geq q_2\geq \cdots  \geq q_n $. The Laplacian spectrum is defined analogously. The multiset of eigenvalues of $Q_G$ (resp. 
$L_G$, $A_G$) is called the $Q$-spectrum (resp. $L$-spectrum, $A$-spectrum) of $G$. For any bipartite graph, its $Q$-spectrum
coincides with its $L$-spectrum. Two graphs are $Q$-cospectral (resp. $L$-cospectral, $A$-cospectral) if they have the
same $Q$-spectrum (resp. $L$-spectrum, $A$-spectrum). A graph $G$ is said to be $DQS$ (resp. $DLS$, $DAS$) if there is
no other non-isomorphic graph $Q$-cospectral (resp. $L$-cospectral, $A$-cospectral) with $G$.
\vspace{2mm}
Up to now, only some graphs with special structures are shown to be {\it determined by their spectra} (DS, for short) (see \cite{WYY, A, AA, AAA, AAAA, CHA, CRS1, CS, Das, HLZ, LP, Mer, M, W1} and the references cited in them). \vspace{1mm}
Van Dam and Haemers \cite{VH} conjectured that almost all graphs are determined by their spectra. Nevertheless, the set of graphs that are known to be detrmined by their spectra is too small. So, discovering infinite classes of graphs that are determined by their spectra can be an interesting problem. About the background of the question " Which graphs are determined by their spectrum?", we refer to \cite{VH}. In this work, we show that the multicone graphs
$ K_w\bigtriangledown P$ are determined  by their  signless Laplacian spectrum, their complement with respect to signless Laplacian spectra,  their adjacency spectrum and their Laplacian spectrum.\\

Julius Petersen (1839-1910) was a Danish mathematician. Around 1898 he constructed the graph bearing his name as the smallest counterexample against the claim that a connected bridgeless cubic graph has an edge colouring with three colours. The Petersen graph is an undirected graph with 10 vertices and 15 edges. It is a small graph that serves as a useful example and counterexample for many problems in graph theory. The Petersen graph is named after Julius Petersen, who in 1898 constructed it to be the smallest bridgeless cubic graph with no three-edge-coloring \cite{NB}.\\

This paper is organized as follows. In Section \ref{2}, we review some basic information and preliminaries. Then in Section \ref{3}, we show that any graph $ Q $-spectral with a (complement of) multicone graph $ K_w\bigtriangledown P $ is DS. In Section \ref{4}, we show that these graphs are determined by their Laplacian spectrum. In Section \ref{5}, we prove that these graphs are determined by their adjacency spectrum.

\section{Some definitions and preliminaries}\label{2}

In this section, we recall some results that play an important role throughout this paper.

\begin{lemma}[\cite{CHA}]\label{lem 2-1} Let $ G $ be a graph with $ n $ vertices, $ m $ edges, $ t $ triangles and the vertex degrees $ d_1, d_2, \cdots  , d_n $. If
${T_k} = \sum\limits_{i = 1}^n {{{({q_i}(G))}^k}}$, then
$$T_0=n,\quad T_1=\sum\limits_{i = 1}^n d_i=2m,\quad T_2=2m+\sum\limits_{i = 1}^n d^2_i, \quad T_3=6t+3\sum\limits_{i = 1}^n d^2_i+\sum\limits_{i = 1}^n d^3_i.$$
\end{lemma}

\begin{lemma}[\cite{CRS1}]\label{lem 2-2}
In any graph the multiplicity of the eigenvalue 0 of the signless Laplacian is equal
to the number of bipartite components.
\end{lemma}
\begin{lemma}[\cite{LP}]\label{lem 2-3}
Let $G$ be an $r$-regular graph on $n$ vertices and $G$ is determined
by its signless Laplacian spectrum. Let $H$ be a graph $Q$-cospectral with $G\bigtriangledown K_m$. If $d_1(H) = d_2(H) = \cdots  = d_m(H) = n + m-1$, then $ H\cong K_m\bigtriangledown G $.
\end{lemma}

\begin{lemma}[\cite{LP, CHA}]\label{lem 2-4}
Let $G$ be a connected graph of order $n$ ($n> 1$) with the minimum degree $\delta $. Then $ q_n<\delta$.
\end{lemma}

\begin{lemma}[\cite{LP}]\label{lem 2-5}
For $i = 1, 2$, let $G_i$ be an $r_i$-regular
graph on $n_i$ vertices. Then
\begin{center}
$ P_{Q(G_1\bigtriangledown G_2)}(x)=\dfrac{P_{Q(G_1)}(x-n_2)P_{Q(G_2)}(x-n_1)}{(x-2r_1-n_2)(x-2r_2-n_1)}f(x)$,
\end{center}
where $ f(x) = x^2 - (2(r_1 + r_2) + (n_1 + n_2))x + 2(2r_1r_2 + r_1n_1 + r_2n_2). $
\end{lemma}

\begin{lemma}[\cite{CHA}]\label{lem 2-6}
Let $G$ be a graph on $n$ vertices with the vertex degrees
$d_1, d_2,\cdots, d_n$. Then
$$min\left\{ {d_i + d_j} \right\}\leq q_1\leq max\left\{ {d_i + d_j}\right\},$$
where $(i, j)$ runs over all pairs of adjacent vertices of $G$.
\end{lemma}

\begin{lemma}[\cite{CHA}]\label{lem 2-7}
Let $G$ be a graph with maximum degree $d_1$ and second maximum degree $d_2$. Then
$ q_2(G)\geq d_2-1 $. If $q_2(G) = d_2-1$, then $d_1 = d_2$.
\end{lemma}

A connected bipartite graph is called balanced if the sizes of its vertex classes are equal, and unbalanced otherwise. An isolated vertex is considered to be an unbalanced bipartite graph (see \cite{FIL}).

\begin{lemma}[\cite{FIL}]\label{lem 2-90}
Let $G$ be a graph of order $n>2$. Then $q_2(G)\leq n-2$ and $ q_1(G)\geq q_2(G)\geq . . .\geq q_n(G) $. Moreover, $q_{k+1}(G) =n-2$ ($ 1\leq k<n$) if and only if $ \overline{G}$ has either $ k $ balanced bipartite components or $k + 1$ bipartite components.
\end{lemma}

\begin{lemma}[\cite{CRS1}]\label{lem 91}
In bipartite graphs the $Q$-polynomial (the characteristic polynomial
of the signless Laplacian matrix) is equal to the characteristic polynomial
of the Laplacian matrix.
\end{lemma}

\begin{lemma}[\cite{WYY}]\label{lem 92}
A connected bipartite graph $ G $ has three distinct eigenvalues if and only if it is either a complete regular bipartite graph or a star. In this case $ G $ is Laplacian integral, i.e., the Laplacian eigenvalues of $ G $ are integral.
\end{lemma}

\begin{proposition}[\cite{FIL}]\label{prop 80}
Let $ G $ be an $ r $-regular graph on $ n $ vertices and let $ \overline{q_1}\geq \overline{q_1}\geq \cdots  \geq \overline{q_n}$ be the signless Laplacian eigenvalues of $ Q(\overline{G})$. Then $ \overline{q_1}=2(n-r-1)$ and $ \overline{q_i}=n-2-q_{n-i+2}$ for $ i=2, 3, \cdots, n$, where $ q_i$'s denote the signless Laplacian eigenvalues of $ G $.
\end{proposition}

\begin{lemma}[\cite{A,AA,AAA,AAAA,B,M,VH}]\label{lem 1}
Let $ G $ be a graph. For the adjacency matrix and the Laplacian matrix, the following information is obtained from the spectrum:
\begin{flushleft}
$(i)$ The number of vertices,

$(ii)$ The number of edges.

For the adjacency matrix, the following follows from the spectrum:

$(iii)$ The number of closed walks of any length,

$(iv)$ Being regular or not and the degree of regularity,

$(v)$ Being bipartite or not.

For the Laplacian matrix, the following follows from the spectrum:

$(vi)$ The number of spanning trees,

$(vii)$ The number of components,

$(viii)$ The sum of squares of degrees of vertices.
\end{flushleft}
\end{lemma}

The adjacency spectrum of the Petersen graph $P$ is as follows:
${\rm{Spec}}_{A}(P)=\left\{ {{{\left[ 3 \right]}^1},\left[ 1 \right]}^{5},\,{{\left[ -2 \right]}^{4}} \right\} $ (see \cite{B}).

\begin{theorem}[\cite{A,AA,AAA,AAAA,CRS,M}]\label{the 1} If $G_1$ is $r_1$-regular with $n_1$ vertices, and $G_2$ is $r_2$-regular with $n_2$ vertices, then the characteristic polynomial of the join $ G_1\bigtriangledown G_2 $ is given by:

\begin{center}
$P_{G_{1}\bigtriangledown G _{2}}(y)=\frac{P_{G_{1}}(y)P_{G_{2}}(y)}{(y-r_1)(y-r_2)}((y-r_1)(y-r_2)-n_1n_2)$.
\end{center}
\end{theorem}

The spectral radius of a graph $ \Lambda $ is the largest eigenvalue of its adjacency matrix and it is denoted by $ \varrho(\Lambda) $. A graph is called bidegreed, if the set of degrees of its vertices consists only of two elements.\vspace{2mm}

Theorem \ref{the 2} gives an upper bound on the spectral radius. For further information about this inequality we refer the reader to \cite{W1} (see the first paragraph after Corollary 2.2 and also Theorem 2.1 of \cite{W1}). In \cite{W1}, it is stated that  if $ G $ is disconnected, then the equality holds. However, in this paper we only consider {\it connected cases} and we state the equality in this case.

\begin{theorem}[\cite{A, AA, AAA, AAAA, M, HLZ, W1}]\label{the 2}
Let $ G $ be a simple graph with $ n$ vertices and $ m $ edges. Let $ \delta=\delta(G)$ be the minimum degree of vertices of $ G $ and $\varrho(G)$ be the spectral radius of the adjacency matrix of $ G $. Then
\begin{center}
$ \varrho(G)\leq \frac{\delta-1}{2}+\sqrt{2m-n\delta+\frac{(\delta+1)^{2}}{4}} $.

\end{center}
Equality holds if and only if $ G $ is either a regular graph or a bidegreed graph in which each vertex is of degree either $ \delta $ or $ n-1 $.
\end{theorem}

\begin{theorem}[\cite{A, AA, AAA, AAAA, M, Mer}] \label{the 3} Let $ G $ and $ H $ be two graphs with the Laplacian spectrum
$ \lambda_1\ge\lambda_2\ge \cdots\ge\lambda_n $ and $ \mu_1\ge\mu_2\ge \cdots\ge\mu_m $, respectively. Then the Laplacian spectrum of $ \overline{G}\ $ and $G \bigtriangledown H $ are $ n -\mathop \lambda \nolimits_1 ,n - \mathop \lambda \nolimits_2 ,\cdots,n - \mathop \lambda \nolimits_{n - 1} ,0 $ and $ n + m,m + \mathop \lambda \nolimits_1 ,\cdots,m + \mathop \lambda \nolimits_{n - 1} ,n + \mathop \mu \nolimits_1 ,\cdots, n + \mathop \mu \nolimits_{m - 1} ,0$, respectively.
\end{theorem}

\begin{theorem}[\cite{A, AA, AAA, AAAA, M, Mer}]\label{the 4} Let $ G $ be a graph on $n$ vertices. Then $n$ is one of the Laplacian eigenvalues of $ G $ if and only if $ G $ is the join of two graphs.
\end{theorem}

\begin{theorem}[\cite{A,AA,AAA,AAAA, M}]\label{the 5} For a graph $G$, the following statements are equivalent:
\begin{flushleft}
$(i)$ $G$ is regular.

$(ii)$ $ \varrho(G)=d_{G} $, the average vertex degree.

$(iii)$ $G$ has $ v=(1,1,\cdots,1)^T $ as an eigenvector for $ \varrho(G)$.

\end{flushleft}
\end{theorem}

\section{Main Results}

\subsection{Connected graphs $ Q $-spectral with a (complement of)  multicone graph $ K_w\bigtriangledown P $}\label{3}

In what follows, we always suppose that $ \Delta=d_1\geq d_2\geq \cdots  \geq d_{n}=\delta $ and $ q_1\geq q_2\geq \cdots  \geq q_{n} $.
\begin{proposition}
The signless Laplacian spectrum of the multicone graph $ K_1\bigtriangledown P $ is:

\begin{center}
$\left\{ {{{{\left[ 5 \right]}^6},\,{{\left[ 3 \right]}^4}},\,{{\left[ 12 \right]}^1}}\right\}$.
\end{center}
\end{proposition}
\begin{proof}
By Lemma \ref{lem 2-5} the result follows (see also Theorem 3.1 of \cite{CHA}).
\end{proof}

\begin{theorem}\label{the 3-1}
The multicone graph  $ K_1\bigtriangledown P $ is DS with respect to its signless Laplacian spectra.
\end{theorem}
\begin{proof}
It is easy and straightforward to see that there is no disconnected graph which is $ Q $-spectral with the multicone graph  $ K_1\bigtriangledown P $. Otherwise, let ${\rm{Spec}}_{Q}(G)={\rm{Spec}}_{Q}(K_1\bigtriangledown P)$ and $G=H_1\cup H_2 $, where $ H_i $'s  ($i=1, 2$) are subgraphs of $ G $. It is easy to check that any of $ H_i $'s must have three signless Laplacian eigenvalues. It is well-known that a graph has one or two signless Laplacian eigenvalue(s) if and only if it is either isomorphic to a disjoint union of isolated vertices or a disjoint union of complete graphs on the same vertices, respectively. But, by the spectrum of $ G $ and ${\rm{Spec}}_{Q}(K_w)=\left\{ {{{{\left[ 2w-2 \right]}^1},\,{{\left[ w-2 \right]}^{w-1}}}}\right\}$ this case (having three distinct signless Laplacian eigenvalues for any of subgraphs $ H_1 $ and $ H_2 $) does not happen. So, any graph which is 
$Q$-spectral with the multicone graph
$ K_1\bigtriangledown P $, the cone of the Petersen graph, is connected. By Lemma \ref{lem 2-6}  $ 2d_1\geq q_1=12 $. So, $ d_1=\Delta\geq 6 $. Also, it follows from Lemma \ref{lem 2-4} that $ \delta>q_{11}=3 $. So, $ \delta\geq 4 $. By Lemma \ref{lem 2-7} $ d_2\leq 6$. Hence we conclude that $ 4\leq \delta=d_{11}\leq d_2\leq 6 $. Now let us consider the following cases:\\

Case 1. $ d_2=4 $.\\
Hence $ d_2=d_3=\cdots=d_{11}=4$. Therefore, by Lemma \ref{lem 2-1} $ d_1+40=50 $ and so $d_1=10$. Now, by Lemma \ref{lem 2-3} the result follows.\\

Case 2. $ d_2=5 $.\\
Suppose that we have $ a $ vertices of degree $ 4 $ and $ 10-a $ vertices of degree $ 5 $ in $ d_i $'s ($ 2\leq i\leq 11$). Therefore, $ d_1+4a+(10-a)5=50 $. This implies that $d_1-a=0 $ and so $ a=d_1 $. But $ 10 \geq d_1\geq 6 $ and so  $ a=d_1\in \left\{ {6,7,8,9, 10} \right\}$. Now, we consider the following cases:\\
\begin{center}
$\begin{cases}
d_1=a= 6,\\
(4)^1, (5)^4, (6)^1\\
\end{cases}$, $\begin{cases}
d_1=a = 7,\\
(7)^1, (4)^7, (5)^3\\
\end{cases}$, $\begin{cases}
d_1=a = 8,\\
(8)^1, (4)^8, (5)^2\\
\end{cases}$,
$\begin{cases}
d_1=a = 9,\\
(9)^1, (4)^9, (5)^1\\
\end{cases}$, $\begin{cases}
d_1=a= 10,\\
(10)^1, (4)^{10}, (5)^0,\\
\end{cases}$
\end{center}
where, for example $ (4)^1, (5)^4, (6)^1 $ denotes the vertex degrees of any graph $ Q $-spectral with the multicone graph $ K_1\bigtriangledown P $ and $ (4)^1 $ means that one vertex of degree 4. But, all the above cases contradict the fact that
$ T_2=2m+\sum\limits_{i = 1}^n d^2_i $ (see Lemma \ref{lem 2-1}). By the way, obviously the last case does not happen, since $ d_2=5 $.\\

Case 3. $ d_2=6 $.\\
Let we have $ a $ vertices of degree $ 4 $, $ b $ vertices of degree $ 5 $ and $ c $ vertices of degree $ 6 $ in $ d_i $'s ($ 2\leq i\leq 10 $). So, by Lemma \ref{lem 2-1} we get:\\

\begin{center}
 $\begin{cases}
a + b + c = 10,\\
4a +5b + 6c + d_1 = 50,\\
16a + 25b + 36c + d^2_1 =260.
\end{cases}$
\end{center}
By a simple calculating we get $ a=\dfrac{-d^2_1+11d_1+10}{2} $, $ b=-10d_1+d^2_1 $ and $ c=\dfrac{10+9d_1-d^2_1}{2} $. Now, if $ d_1\in \{6, 7, 8, 9\}$ we will have a contradiction to the fact that $ 0 \leq a, b, c\leq 10 $ or $ 0 \leq a+b+c=10 $. If $ d_1=10 $, then $ a=10 $ and so $ b=c=0 $. This means that the vertex degrees of any graph which is $ Q $-spectral with the multicone graph $ K_1\bigtriangledown P $ is $ (4)^{10}, (10)^1 $, a contradiction to the fact that $ d_2=6 $.
\end{proof}

  \begin{proposition}
The signless Laplacian spectrum of the multicone graph $ K_2\bigtriangledown P $ is:

\begin{center}
$\left\{ {{{{\left[ \dfrac{20\pm\sqrt{48}}{2} \right]}^1},\,{{\left[ 10 \right]}^1}},\,{{\left[ 6 \right]}^5}},\,{{\left[ 3 \right]}^4}\right\}$.
\end{center}
\end{proposition}
\begin{proof}
By Lemma \ref{lem 2-5} the result follows (see also Corollary 3.1 of \cite{LC}).
\end{proof}

  \begin{theorem}\label{the 3-2}
The multicone graph  $ K_2\bigtriangledown P $ is DS with respect to its signless Laplacian spectra.
\end{theorem}
\begin{proof}

Let $ G $ be $ Q $-spectral with the multicone graph $ K_2\bigtriangledown P $. By Lemma \ref{lem 2-6} we deduce that $q_1(G)\geq 13.64 $. So, $ 2d_1\geq q_1(G)\geq 14 $. Therefore, $ d_1\geq 7 $. On the other hand, it follows from  Lemma \ref{lem 2-4} that  $ \delta> q_{12}=3 $ (It is straightforward to see that any graph which is
$ Q $-spectral with the multicone graph $ K_2\bigtriangledown P $ is connected). By \cite[Lemma 2.7]{Das1} $ d_3\leq 8.4 $ and so $4 \leq\delta \leq d_3\leq 8 $. Now, we consider the following cases:\\

Case 1. $ d_3=4 $.\\
Take $d_1+d_2=x$ and $d^2_1+d^2_2=y$. It is clear that $8 \leq 2\delta \leq x\leq 2\Delta\leq 22 $. In this case, $ d_3=d_4= \cdots= d_{12}=4$. Therefore, $ 40+x=72 $ and so $ x=32 $, a contradiction.\\

Case 2.  $ d_3=5 $.\\
Assume that there are $ a $ vertices of degree $4$ and $10-a $ vertices of degree $ 5 $ in $ d_i $'s ($ 4\leq i \leq 11$). So, $ x+4a+(10-a)5=72 $ and so $ x=a+22$. So, we must have $ x=a+22\in \left\{ {8, 9, \cdots, 22} \right\}$ or $ a\in \left\{ {-14, -13,\cdots, 0} \right\}$.  Therefore, $ a=0 $ and $ x=22 $. Therefore, $ d_1=d_2=11 $. Now, by Lemma \ref{lem 2-3} the result follows.\\

Case 3. $ d_3=6 $.\\
Suppose that there exist $ a $ vertices of degree $ 4 $, $b $ vertices of degree $ 5 $ and $ c $ vertices of degree $ 6 $ in $ d_i $'s ($ 3\leq i \leq 12$). So

\begin{center}
$\begin{cases}
a + b + c = 10,\\
4a +5b + 6c = 72-x,\\
16a + 25b + 36c=492-y.
\end{cases}$\\
\end{center}
By a simple computation we have $ a=\dfrac{11x-y}{2}$, $ b=-12-10x+y $, and $ c=22+\dfrac{9x-y}{2} $. It is clear that $ 0 \leq a, b,c \leq 10 $ and $ 0\leq a+b+c= 10 $.
This means that the summation of $ x $ and $ y $ must be even, since $ a $ is a  non-negative integer number. In other words, $ d_1, d_2\in \left\{ {6, 7, 8,9, 10, 11} \right\} $, since $ 11 \geq d_1, d_2\geq d_3=6 $ and $ d_1\geq 7 $. Therefore,

$\begin{cases}
x = 14,\\
y = 100.\\
\end{cases}$, $\begin{cases}
x = 16,\\
y = 136.\\
\end{cases}$, $\begin{cases}
x = 18,\\
y = 164.\\
\end{cases}$,
$\begin{cases}
x = 16,\\
y = 128.\\
\end{cases}$, $\begin{cases}
x = 20,\\
y = 200.\\
\end{cases}$, $\begin{cases}
x = 14,\\
y = 98.\\
\end{cases}$, $\begin{cases}
x = 16,\\
y = 130.\\
\end{cases}$, $\begin{cases}
x = 18,\\
y = 170.\\
\end{cases}$,  $\begin{cases}
x = 18,\\
y = 162.\\
\end{cases}$, $\begin{cases}
x = 20,\\
y = 202.\\
\end{cases}$, $\begin{cases}
x = 22,\\
y = 242.\\
\end{cases}$\\
 By replacing any of the above cases we have a contradiction to $ 0\leq a\leq 10 $. If the case $\begin{cases}
x = 22,\\
y = 242.\\
\end{cases}$ happens, then $ a=0 $ and $ c=-10 $, a contradiction.\\

Case 4. $ d_3=7 $.\\
Suppose that there are $ a $ vertices of degree $ 4 $,  $b $ vertices of degree $ 5 $, $ c $ vertices of degree $ 6 $ and $ d $ vertices of degree $ 7 $ in $ d_i $'s ($ 4\leq i \leq 11$). So

\[
\begin{cases}
a + b + c+d= 10,\\
4a +5b + 6c = 72-x-7d,\\
16a + 25b + 36c=492-y-49d.
\end{cases}
\]
It follows that $ a=\dfrac{11x-y}{2}-d$, $ b=-12-10x+y+3d $ and
$ c=22-3d+\dfrac{9x-y}{2} $. It is clear that $ e\geq 1 $. So, both $ x $ and $ y $ must
simultaneously be  either even or odd, since $ a $ is an integer. This leads to a contradiction in the same way as the case $ d_3=6$.\\

Case 5. $ d_3=8 $.\\
Suppose that there are $ a $ vertices of degree $ 4$,  $b $ vertices of degree $ 5 $, $ c $ vertices of degree $ 6 $, $ d $ vertices of degree $ 7 $ and $ e $ vertices of degree $ 8$ in $ d_i $'s ($ 3\leq i \leq 12$). So\\

\begin{center}
$\begin{cases}
a + b + c+d+e= 10,\\
4a +5b + 6c+7d+8e = 72-x,\\
16a + 25b + 36c+49d+64e=492-y.
\end{cases}$\\
\end{center}
By solving the above equations we get $ a=\dfrac{11x-y}{2}-d-2e$, $ b=-12-10x+y+3d+8e $ and $ c=22-3d-6e+\dfrac{9x-y}{2} $. It is clear that $ e\geq 1 $. This leads to a contradiction in the same way as the case $ d_3=6 $, since $ 0 \leq b \leq 10 $.
\end{proof}

Now, we show that  the multicone graphs  $ K_w\bigtriangledown P $ are DS with respect to their signless Laplacian spectra. To do so,  we need one lemma.\\

\begin{lemma}\label{Cor 3-1} Let $ G $ be a graph of order $ n $. If $ n-2 $ is one of the signless Laplacian eigenvalues of $ G $ with the multiplicity of at least $ 2 $ and $ q_1(G)> n-2 $, then $ G $ is the join of two graphs.
\end{lemma}
\begin{proof}
By Lemma \ref{lem 2-90} $\overline{G}$ has either at least $ 2 $ balanced bipartite components or at least $3$ bipartite components. In other words, $\overline{G}$ is disconnected. Thus $ G $ is connected and it is the join of two graphs.
\end{proof}

\begin{remark} 
Note that in Lemma \ref{Cor 3-1} the condition $ q_1(G)> n-2  $  is critical. For instance, consider the graph which is the disjoint union of two triangles. It is easy to see that the signless Laplacian spectrum of this graph is $\left\{ {{{\left[ 4 \right]}^2},\,{{\left[ 1 \right]}^4}} \right\}$, whereas this graph is not the join of any two graphs.
\end{remark}
\begin{theorem}
Multicone graphs  $ K_w\bigtriangledown P $ are DS with respect to their signless Laplacian spectrum.
\end{theorem}

\begin{proof}
We perform mathematical induction on $ w $. For $ w=1, 2 $ this theorem was proved (see Theorem \ref{the 3-1} and Theorem \ref{the 3-2}). Let the theorem be true for $ w $; that is, if $ {\rm{Spec}}_{Q}(H)={\rm{Spec}}_{Q}(K_w\bigtriangledown P)$, then $ H\cong K_w\bigtriangledown P $, indeed $ H $ is an arbitrary graph $ Q $-spectral with a multicone graph $ K_w\bigtriangledown P$ (the inductive hypothesis). We show that it follows from 
$ {\rm{Spec}}_{Q}(G)={\rm{Spec}}_{Q}(K_{w+1}\bigtriangledown P)$ that $ G\cong K_{w+1}\bigtriangledown P $. It is clear that $ G $ has one vertex and $ 10+w $ edges more than $ H $ and $ {\rm{Spec}}_{Q}(G)={\rm{Spec}}_{Q}(K_1\bigtriangledown H)$ (by the inductive hypothesis $ H\cong K_w\bigtriangledown P$). On the other hand, by Lemma \ref{Cor 3-1}, $G$ and $ H $ are the join of two graphs. So, we must have $ G\cong K_1\bigtriangledown H $. By  the inductive hypothesis the proof is completed.
\end{proof}

\begin{theorem}\label{the 8}
The complement of multicone graphs  $K_w\bigtriangledown P$ are DS with respect to their signless Laplacian spectrum.
\end{theorem}
\begin{proof}
 Let $\Gamma$ be a graph such that $ {\rm{Spec}}_{Q}(\Gamma)= {\rm{Spec}}_{Q}(\overline{K_w\bigtriangledown P})$. Then we show that $ \Gamma \cong \overline{K_w\bigtriangledown P}$.
It is clear that
\[{\rm{Spec}}_{Q}(\Gamma)={\rm{Spec}}_{Q}(\overline{K_w\bigtriangledown P})={\rm{Spec}}_{Q}(\overline{K_w}\cup\overline{P})=
{\rm{Spec}}_{Q}(wK_1\cup \overline{P}).\]
 It is straightforward to check that if $ G $ is an $ r $-regular graph, then
 $ P_{Q(G)}(x)=P_{A(G)}(x-r)$. So,
$ {\rm{Spec}}_{Q}(P)=\left\{ {{{\left[ 6 \right]}^1},\left[ 4 \right]}^{5},\,{{\left[1 \right]}^{4}} \right\} $. By Proposition \ref{prop 80} we deduce that $ {\rm{Spec}}_{Q}(\overline{P})=\left\{ {{{\left[ 12 \right]}^1},\left[ 4 \right]}^{5},\,{{\left[7 \right]}^{4}} \right\} $. Therefore, $ {\rm{Spec}}_{Q}(\Gamma)={\rm{Spec}}_{Q}(\overline{K_w\bigtriangledown P})=\left\{ {{{\left[ 12 \right]}^1},\left[ 4 \right]}^{5},\left[ 0 \right]^{w},\,{{\left[7 \right]}^{4}} \right\}$. It follows from Lemma \ref{lem 2-2} that $ \Gamma $ has at least one bipartite component. If $ w\geq 2 $, then $\Gamma$ is disconnected. Hence,  suppose that $ w=1 $. We show that in this case $ \Gamma $ is also disconnected. Now we claim that $ \Gamma $ is not a bipartite graph. Otherwise, by Lemma \ref{lem 91} $ {\rm{Spec}}_{Q}(\Gamma)={\rm{Spec}}_{L}(\Gamma) $. In other words, $ {\rm{Spec}}_{Q}(\overline{K_1\bigtriangledown P})={\rm{Spec}}_{L}(\overline{K_1\bigtriangledown P}) $. But, $ {\rm{Spec}}_{L}(\overline{K_1\bigtriangledown P})=\left\{ {{\left[0 \right]}^{2},\left[ 9 \right]}^{5},\,{{\left[6\right]}^{4}} \right\} $ (see Theorem \ref{the 3}) and so ${\rm{Spec}}_{L}(\overline{K_1\bigtriangledown P})\neq {\rm{Spec}}_{Q}(\overline{K_1\bigtriangledown P}) $. Hence $ {\rm{Spec}}_{L}(\Gamma)\neq{\rm{Spec}}_{Q}(\Gamma) $. This means that $\Gamma$ is not a bipartite graph. Hence we deduce that $ \Gamma $ is disconnected. Now, we show that the bipartite component of the $\Gamma$ is $ K_1 $(an isolated vertex).
Let  $\Gamma$ has a bipartite component, say $ U $. We consider the following cases:\\ 

{\bf Case 1.} {\it Bipartite graph $ U $ have two distinct Laplacian eigenvalues.}\\
In this case, $ U $ is a bipartite graph and also it is a complete graph (It is known that $ G $ has exactly 2 distinct Laplacian eigenvalues if and only if $ G $ is a complete graph). Therefore, $ U\cong K_2 $, and so ${\rm{Spec}}_{L}(U)={\rm{Spec}}_{Q}(U)=\left\{ {{{[0]}^1}\,,\,{{\left[ 2 \right]}^1}} \right\} $, a contradiction to ${\rm{Spec}}_{Q}(\overline{K_1\bigtriangledown P})$.\\

{\bf Case 2.} {\it Bipartite graph $ U $ has three distinct Laplacian eigenvalues.}\\
In this case, by Lemma \ref{lem 92} $ U $ is a regular complete bipartite graph or a star. In other words, $ U $ is either isomorphic to $ K_{n,n}$ or $ K_{1,n}$, where $n$ is an arbitrary natural number. But, ${\rm{Spec}}_{L}(K_{n, n})={\rm{Spec}}_{L}(nK_1\bigtriangledown nK_1)=\left\{ {{{\left[ 0 \right]}^1},\,{{\left[ n \right]}^{2(n - 1)}},\,{{\left[ {2n} \right]}^1}\,} \right\}$ and ${\rm{Spec}}_{L}(K_{1, n})={\rm{Spec}}_{L}(K_1\bigtriangledown nK_1)=\left\{ {{{\left[ 0 \right]}^1},\,{{\left[ n+1 \right]}^{1}},\,{{\left[ {1} \right]}^{n-1}}} \right\}$, a contradiction to ${\rm{Spec}}_{Q}(\overline{K_1\bigtriangledown P})$. \\

{\bf Case 3.} {\it Bipartite graph $ U $ has four distinct Laplacian eigenvalues.} \\
In this case $\Gamma=U\cup K $, where $ {\rm{Spec}}_{L}(U)={\rm{Spec}}_{Q}(U)=\left\{ {{{\left[ 12 \right]}^1},\left[ 4 \right]}^{\alpha},\left[ 0 \right]^{1},\,{{\left[7 \right]}^{\beta}} \right\} $ and ${\rm{Spec}}_{Q}(K)=\left\{ {\left[ 4 \right]}^{\gamma},\,{{\left[7 \right]}^{\kappa}} \right\} $, where $ \alpha $, $ \beta $, $\gamma $, $\kappa$ are natural numbers, $ \alpha\lvertneqq 5 $
, $ \beta\lvertneqq 4 $. But, if a graph has two distinct signless Laplacian eigenvalues, then it is a complete gaph or disjoint union of complete graphs. On the other hand, ${\rm{Spec}}_{Q}(K_w)=\left\{ {\left[ 2w-2 \right]}^{1},\,{{\left[w-2 \right]}^{w-1}} \right\} $, a contradiction to ${\rm{Spec}}_{Q}(K)=\left\{ {\left[ 4 \right]}^{\gamma},\,{{\left[7 \right]}^{\kappa}} \right\} $.\\

So, we conclude that $ U $, the bipartite component of $ \Gamma $, is nothing but $ K_1 $. In other words, $ U $ must have one distinct signless Laplacian eigenvalue. In a similar way of the previous argument one may conclude that for every $ w $, any bipartite component of 
$ \Gamma$ is $ K_1 $. Now, since $\Gamma $ has $ w $ bipartite components, so  $ \Gamma=wK_1 \cup H $, where $ {\rm{Spec}}_{Q}(H)=\left\{ {{{\left[ 12 \right]}^1},\left[ 4 \right]}^{5},\left[ 0 \right]^{w},\,{{\left[7 \right]}^{4}} \right\}
 $. But,
\[{\rm{Spec}}_{Q}(H)={\rm{Spec}}_{Q}(\overline{P})=
\left\{\left[ 12 \right]^1,\left[ 4 \right]^{5},\,{{\left[7 \right]}^{4}} \right\}.\]
It is clear that $ P_{Q(H)}(x)=P_{Q(\overline{P})}(x) =P_{A(\overline{P})}(x-6) $ and so $ P_{Q(H)}(x+6)=P_{Q(\overline{P})}(x+6) =P_{A(\overline{P})}(x) $. Therefore, $H\cong \overline{P} $ (see  \cite[Proposition 3]{VH}). This implies that $\Gamma=wK_1\cup \overline{P}$. Consequently, $\Gamma\cong \overline{K_w\bigtriangledown P}$. This completes the proof.
\end{proof}

\begin{flushright}
In the following, we show that multicone graph $K_w\bigtriangledown P $ is DS with respect to its Laplacian spectrum.
\end{flushright}
\section{The spectral determinations of the multicone graphs $K_w\bigtriangledown P $ with respect to Laplacian spectrum.}\label{4}

\begin{theorem} Multicone graphs $ K_w\bigtriangledown P $ are DS with respect to their Laplacian spectrum.
\end{theorem}
\begin{proof}
 We solve the theorem by induction on $ w $. For $ w=1 $, the proof is clear (see Theorem \ref{the 3} and Theorem \ref{the 4}). Let the claim be true for $ w $; that is, let
 $ G_1 $ be a graph such that  
\[\operatorname{Spec}_{L}(G_1)=\operatorname{Spec}_{L}(K_w\bigtriangledown P)=\left\{ {{{\left[ {w + 10} \right]}^{w}},{{\left[ {5+w} \right]}^{4}},{{\left[ 0 \right]}^1},{{\left[ 2+w \right]}^{5}}} \right\}.\] 
Then $G_1\cong K_w\bigtriangledown P$.  
We show that if 
$$ \operatorname{Spec}_{L}(G)=\operatorname{Spec}_{L}(K_{w+1}\bigtriangledown P)=
\left\{ {{{\left[ {w+1+ 10} \right]}^{w+1}},{{\left[ {6+w} \right]}^{4}},{{\left[ 0 \right]}^1},{{\left[ 3+w \right]}^{5}}} \right\} ,$$ then 
$ G\cong K_{w+1}\bigtriangledown P $. It follows from Theorem \ref{the 4} that $ G_1 $ and $ G $ are join of two graphs. On the other hand, $ G $ has one vertex and $ 10+w $ edges more than $ G_1 $. Hence we must have $ G\cong K_1\bigtriangledown G_1 $.  Now,  the result follows from the induction hypothesis.
\end{proof}

\section{The spectral determinations of the multicone graphs $K_w\bigtriangledown P $ with respect to adjacency spectrum}\label{5}

\begin{proposition}\label{prop 5-1}
The adjacency spectrum of the multicone graphs $ K_w\bigtriangledown P $ is:
\end{proposition}

\begin{center}
$\left\{ {{{\left[ { - 1} \right]}^{w - 1}},\,{{\left[ 1 \right]}^5},\,{{\left[ { - 2} \right]}^4},\,{{\left[ {\dfrac{{w + 2 \pm \sqrt {{w^2} + 32w + 16} }}{2}} \right]}^1}} \right\}$
\end{center}
\begin{proof}
By Lemma \ref{lem 1} and Theorem \ref{the 1} the result follows.
$ \Box $
\end{proof}

\begin{lemma}\label{9} Let $ G $ be cospectral with a multicone graph $ K_w\bigtriangledown P $. Then $ \delta(G) = w+3$.
\end{lemma}

\begin{proof}
Assume that $ \delta(G) = w+8+x$, where $ x $ is an integer. First, it is clear that in this case the equality in Theorem \ref{the 2} happens, if and only if $ x=0 $. We show that $ x=0 $. We suppose by contrary that $ x\neq 0 $. It follows from Theorem \ref{the 2} together with Proposition \ref{prop 5-1}  that
\begin{align*}
\varrho(G) &= \frac{w+2+ \sqrt{8k-4l(w+3)+(w+4)^{2}}}{2}\\
&< \frac{w+2+x+ \sqrt{8k-4l(w+3)+(w+4)^{2}+x^{2}+(2w+8-4l)x }}{2} ,
\end{align*}
where the integer numbers $ k $ and $ l $ denote the number of edges and the number of vertices of the graph $ G $, respectively. For convenience, we let $B = 8k-4l(w+3)+(w+4)^{2}\geq 0 $ and $C = w+4-2l$, and also let $g(x) =x^{2}+(2w+8-4l)x = x^2 + 2Cx.$ Then clearly

\begin{center}
$ \sqrt{B}-\sqrt{B+g(x)} < x. $
\end{center}

We consider two cases:

{\bf Case 1.} $x < 0$.\\
It is easy and straightforward to see that $ |\sqrt{B}-\sqrt{B+g(x)}| >|x| $, since $ x<0 $.\\
\noindent Transposing and squaring yields

\begin{center}
$2B+g(x)-2\sqrt{B(B+g(x))} > x^{2}.$
\end{center}

\noindent Replacing $g(x)$ by $x^2 + 2Cx$, we get

\begin{center}
$B+Cx > \sqrt{B(B+x^2 + 2Cx)}.$
\end{center}

\noindent Obviously $ Cx\geq 0 $. Squaring again and simplifying yields

\begin{center}
$C^{2} > B. $
\end{center}
Therefore $k<\frac{l(l-1)}{2}$.
So, if $ x<0 $, then $ G $ cannot be a complete graph. In other words, if $ G $ is a complete graph, then $ x> 0 $. Or one can say that if $ G $ is a complete graph, then:

\begin{center}
$ \delta(G) > w+3 $  ($ \dagger $).
\end{center}

{\bf Case 2.} $ x>0$.\\
In the same way of {\bf Case 1}, we can conclude that if $ G $ is a complete graph, then:

\begin{center}
$\delta(G) < w+3$ ($ \ddagger $).
\end{center}
But, Eqs. ($ \dagger$) and ($ \ddagger $) does not happen in the same time. Hence we must have $ x=0 $.
Therefore, the assertion holds.
\end{proof}

\vspace{5mm}
\begin{flushleft}
In the next lemma, we show that any graph cospectral with a multicone graph $ K_w\bigtriangledown P $ must be bidegreed.
\end{flushleft}
\begin{lemma}\label{100} Let $ G $ be cospectral with a multicone graph $ K_w\bigtriangledown P$. Then $ G $ is bidegreed in which any vertex of $ G $ is of degree $ w+3$ or $ w+9$.
\end{lemma}
\begin{proof}
 By Theorem \ref{the 5}, $ G $ is not regular. Now, the result follows from Lemma \ref{9} together with Theorem \ref{the 2}.
\end{proof}

\begin{theorem}
The multicone graphs $ K_w\bigtriangledown P $ are determined by their adjacency spectra.
\end{theorem}
\begin{proof}
Let $ {\rm{Spec}}_{A}(\Gamma)= {\rm{Spec}}_{A}(K_w\bigtriangledown P) $. By Lemma \ref{100}, $ \Gamma $ has a graph  as its subgraph in which the degree of any its vertex is $ w+8$ ($ w $ vertices of $ \Gamma $ is of degree $ w+8 $). In other words, $ \Gamma\cong K_w\bigtriangledown H $, where $ H $ is a subgraph of $ G $. Now, we remove all vertices of $ K_w $ and we consider another $ 10 $  vertices. The degree of regularity the obtained graph consisting of these vertices is $ 3 $. Consider $ H $ consisting of these $ 10 $ vertices. $H$ is regular and degree of its regularity is $ 3$. By Theorem \ref{the 1}, $ {\rm{Spec}}_{A}(H)= {\rm{Spec}}_{A}(P)=\left\{ {{{\left[ 3 \right]}^1},\left[ 1 \right]}^{5},\,{{\left[ -2 \right]}^{4}} \right\}$. Hence $ H\cong P$. This completes the proof.
\end{proof}

\end{document}